\documentclass[12pt]{amsart}
 \usepackage{amssymb,amsmath,amsfonts,latexsym,setspace}
 \usepackage{bm}
 \usepackage{array,graphics,color}
 \theoremstyle{plain}
 \newtheorem{propn}{Proposition}[section]
 \newtheorem{thm}[propn]{Theorem}
 \newtheorem{lemma}[propn]{Lemma}
 \newtheorem{cor}[propn]{Corollary}
 \newtheorem*{thm*}{Theorem}
 \theoremstyle{definition}
 \newtheorem{defn}[propn]{Definition}

 \theoremstyle{remark}
 \newtheorem*{rem}{Remark}

 \newcommand{\Hil}{\mathcal{H}}
 \newcommand{\q}{\mathcal{Q}}
 \newcommand{\s}{\mathcal{S}}

\newcommand{\clb}{\mathcal{B}}

\newcommand{\cle}{\mathcal{E}}

\newcommand{\clg}{\mathcal{G}}
\newcommand{\clh}{\mathcal{H}}

\newcommand{\cll}{\mathcal{L}}
\newcommand{\clm}{\mathcal{M}}

\newcommand{\clo}{\mathcal{O}}

\newcommand{\clq}{\mathcal{Q}}

\newcommand{\cls}{\mathcal{S}}

\newcommand{\clw}{\mathcal{W}}

\newcommand{\z}{\bm{z}}
\newcommand{\w}{\bm{w}}
\newcommand{\raro}{\rightarrow}

\newcommand{\Nat}{\mathbb{N}}
\newcommand{\Comp}{\mathbb{C}}

 \newcommand{\D}{\mathbb{D}}
 
 \newcommand{\ot}{\otimes}
 \newcommand{\ovt}{\bar{\ot}}
 \newcommand{\wt}{\widetilde}

\begin{document}

 \title{Tensor product of quotient Hilbert modules}

 \author[Chattopadhyay] {Arup Chattopadhyay}
 \address{%\vskip10pt
 (A. Chattopadhyay) Indian Statistical Institute \\ Statistics and
 Mathematics Unit \\ 8th Mile, Mysore Road \\ Bangalore \\ 560059 \\
 India}

 \email{arup@isibang.ac.in, 2003arupchattopadhyay@gmail.com}

 \author[Das] {B. Krishna Das}

 \address{%\vskip10pt
 (B. K. Das) Indian Statistical Institute \\ Statistics and
 Mathematics Unit \\ 8th Mile, Mysore Road \\ Bangalore \\ 560059 \\
 India}
 \email{dasb@isibang.ac.in, bata436@gmail.com}

  \author[Sarkar]{Jaydeb Sarkar}

 \address{%\vskip10pt
 (J. Sarkar) Indian Statistical Institute \\ Statistics and
 Mathematics Unit \\ 8th Mile, Mysore Road \\ Bangalore \\ 560059 \\
 India}

 \email{jay@isibang.ac.in, jaydeb@gmail.com}


\begin{abstract}
In this paper, we present a unified approach to problems of tensor
product of quotient modules of Hilbert modules over $\mathbb{C}[z]$
and corresponding submodules of reproducing kernel Hilbert modules
over $\mathbb{C}[z_1, \ldots, z_n]$ and the doubly commutativity
property of module multiplication operators by the coordinate
functions. More precisely, for a reproducing kernel Hilbert module
$\clh$ over $\mathbb{C}[z_1, \ldots, z_n]$ of analytic functions on
the polydisc in $\mathbb{C}^n$ which satisfies certain conditions,
we characterize the quotient modules $\q$ of $\clh$ such that $\q$
is of the form $\q_1 \otimes \cdots \otimes \q_n$, for some one
variable quotient modules $\{\q_1, \ldots, \q_n\}$. For $\clh$ the
Hardy module over polydisc $H^2(\mathbb{D}^n)$, this reduces to some
recent results by Izuchi, Nakazi and Seto and the third author. This
is used to obtain a classification of co-doubly commuting submodules
for a class of reproducing kernel Hilbert modules over the unit
polydisc. These results are applied to compute the cross commutators
of co-doubly commuting submodules. This is used to give further
insight into the wandering subspaces and ranks of submodules of the
Hardy module case. Our results includes the case of weighted Bergman
modules over the unit polydisc in $\mathbb{C}^n$.
\end{abstract}

\maketitle

\section{Introduction.}\label{sec: intro}

The question of describing the invariant and co-invariant subspaces
of shift operators on various holomorphic functions spaces is an old
subject that essentially began with the work of A. Beurling
\cite{B}. The analogous problems for holomorphic function spaces in
several variables have been considered in the work by Ahern,
Douglas, Clark, Yang, Guo, Nakazi, Izuchi, Seto and many more (see
\cite{AC}, \cite{ACD}, \cite{DY1}, \cite{douglas2}, \cite{GZ},
\cite{GY}, \cite{bshift}, \cite{SSW}, \cite{Y6}).

In this paper, we will examine certain joint invariant and
co-invariant subspaces of the multiplication operators by the
coordinate functions defined on a class of reproducing kernel
Hilbert spaces on the unit polydisc $\mathbb{D}^n = \{(z_1, \ldots,
z_n) : |z_i| < 1, i = 1, \ldots, n\}$. More precisely, our main
interest is the class of quotient Hilbert modules of reproducing
kernel Hilbert modules over $\mathbb{C}[z_1, \ldots, z_n]$, the ring
of polynomials of $n$ commuting variables, that admit a simple
tensor product representation of quotient modules of Hilbert modules
over $\mathbb{C}[z]$. A related problem also arises in connection
with the submodules and quotient modules of modules over
$\mathbb{C}[z_1, \ldots, z_n]$ in commutative algebra:

\noindent Let $n \in \mathbb{N}$ be a fixed positive integer and
$\{\clm_i\}_{i=1}^n$ be a family of modules over the ring of one
variable polynomials $\mathbb{C}[z]$. Then the vector space tensor
product $\clm:= \clm_1 \otimes_{\mathbb{C}} \cdots
\otimes_{\mathbb{C}} \clm_n$ is a module over $\mathbb{C}[z]
\otimes_{\mathbb{C}} \cdots \otimes_{\mathbb{C}} \mathbb{C}[z] \cong
\mathbb{C}[z_1, \ldots, z_n]$. Here the module action on $\clm$ is
given by \[(p_1 \otimes \cdots \otimes p_n) \cdot (f_1 \otimes
\cdots \otimes f_n) \mapsto p_1 \cdot f_1 \otimes \cdots \otimes p_n
\cdot f_n,\]where $\{p_i\}_{i=1}^n \subseteq \mathbb{C}[z]$ and $f_i
\in \clm_i$ ($1 \leq i \leq n$). Furthermore, let $\clq_i \subseteq
\clm_i$ be a quotient module of $\clm_i$ for each $1 \leq i \leq n$.
Then\begin{equation}\label{T-Q} \clq_1 \otimes_{\mathbb{C}} \cdots
\otimes_{\mathbb{C}} \clq_n,\end{equation} is a quotient module of
$\clm$. Now let $\clq$ be a quotient module and $\s$ be a submodule
of $\clm$. The following question arises naturally in the context of
tensor product of quotient modules.

\noindent \textsf{(A)} when is $\clq$ of the form (\ref{T-Q}).

The next natural question is:

\noindent\textsf{(B)} when is $\clm/\s$ of the form (\ref{T-Q}).

\noindent To the best of our knowledge, this is a mostly unexplored
area at the moment.

Our principal concern in this paper is to provide a complete answer
to the above problem by considering a natural class of reproducing
kernel Hilbert modules over $\mathbb{C}[z]$ replacing the modules in
the algebraic set up. In particular, we prove that a quotient module
$\clq$ of a \textit{standard reproducing kernel Hilbert module} (see
Definition \ref{anal-defn}) over $\mathbb{C}[z_1, \ldots, z_n]$ is
of the form
\[\clq = \clq_1 \otimes \cdots \otimes \clq_n,\]
for $n$ ``one-variable'' quotient modules $\{\q_i\}_{i=1}^n$ if and
only if $\clq$ is \textit{doubly commuting} (see Definition
\ref{DCD}).

The study of the doubly commuting quotient modules, restricted to
the Hardy module over the bidisc $H^2(\mathbb{D}^2)$, was initiated
by Douglas and Yang in \cite{DY1} (also see Berger, Coburn and Lebow
\cite{BCL}). Later in \cite{bshift} Izuchi, Nakazi and Seto obtained
the above classification result only for quotient modules of the
Hardy module $H^2(\mathbb{D}^2)$. More recently, the third author
extended this result to $H^2(\mathbb{D}^n)$ for any $n \geq 2$ (see
\cite{JS1}, \cite{sarkar}).

One of the difficulties in extending the above classification result
from the Hardy module to the setting of a reproducing kernel Hilbert
module $\clh$ is that the module maps $\{M_{z_1}, \ldots, M_{z_n}\}$
on $\clh$, the multiplication operators by the coordinate functions,
are not isometries. This paper overcomes such a difficulty by
exploiting the precise geometric and algebraic structure of tensor
product of reproducing kernel Hilbert modules. In what follows we
develop methods which link the tensor product of Hilbert modules
over $\mathbb{C}[z_1, \ldots, z_n]$ to Hilbert modules over
$\mathbb{C}[z]$.

We also consider the issue of essentially doubly commutativity of
co-doubly commuting submodules of analytic reproducing Hilbert
modules over $\mathbb{C}[z_1, \ldots, z_n]$. We also obtain a
wandering subspace theorem for some co-doubly commuting submodules
of weighed Bergman modules over $\mathbb{C}[z_1, \ldots, z_n]$ and
compute the rank of co-doubly commuting submodules of $H^2(\D^n).$
Our results in this paper are new even in the case of weighted
Bergman spaces over $\mathbb{D}^n$.

We now describe the contents of the paper. After recalling the
notion of reproducing kernel Hilbert modules in Section \ref{sec:1},
we introduce the class of standard Hilbert modules over
$\mathbb{C}[z_1, \ldots, z_n]$ in Section~\ref{sec:2}. Furthermore,
we obtain some basic properties and an useful classification result
for the class of standard Hilbert modules. In Section~\ref{sec:3},
we obtain a characterization of doubly commuting quotient modules of
an analytic Hilbert modules over $\mathbb{C}[\z]$. In Section 4, we
present a characterization result for co-doubly commuting submodules
and compute the cross commutators of a co-doubly commuting
submodule. In section 5, we prove an wandering subspace theorem for
co-doubly commuting submodules of the weighted Bergman modules over
$\D^n$. We also compute the rank of a co-doubly commuting submodules
of $H^2(\D^n)$. The final section is reserved for some concluding
remarks.

\vspace{0.2in}

\noindent\textsf{Notations:}

\begin{itemize}

\item Throughout this paper $n \geq 2$ is a fixed natural number.

\item  For a Hilbert space $\clh$, the set of all bounded linear
operators on $\clh$ is denoted by $\clb(\clh)$.

\item We denote by $\ot$ the Hilbert space tensor product and by
$M\ovt N$, the von-Neumann algebraic tensor product of von-Neumann
algebras $M$ and $N$.

\item For a von-Neumann algebra $M\subseteq \clb(\Hil)$, we denote by
$M'$ the commutant of $M$ that is the von-Neumann algebra of all
operators in $\clb(\Hil)$ which commutes with all the operators in
$M$.

\item For Hilbert space operators $R, T \in \clb(\clh)$, we write
$[R,T] = RT-TR$, the commutator of $R$ and $T$.

\item  For any set $E$, we denote by $\#E$ the cardinality of the
set $E$.

\item For a closed subspace $\cls$ of a Hilbert space $\clh$, we
denote by $P_{\cls}$ the orthogonal projection of $\clh$ onto
$\cls$.

\item For a Hilbert space $\cle$ we shall let $\clo(\mathbb{D}^n,
\cle)$ denote the space of $\cle$-valued holomorphic functions on
$\D^n$.

\item $\mathbb{C}[\z] := \mathbb{C}[z_1, \ldots, z_n]$ denotes the polynomial ring
over $\mathbb{C}$ in $n$ commuting variables

\end{itemize}

\section{Preliminaries}\label{sec:1}

In this section we gather together some known results on reproducing
kernel Hilbert spaces on product domains in $\mathbb{C}^n$. We start
by recalling the notion of a Hilbert module over $\mathbb{C}[\z]$.

Let $\{T_1, \ldots, T_n\}$ be a set of $n$ commuting bounded linear
operators on a Hilbert space $\clh$. Then the $n$-tuple $(T_1,
\ldots, T_n)$ turns $\clh$ into a module over $\mathbb{C}[\bm{z}]$
in the following sense:
\[\mathbb{C}[\bm{z}] \times \clh \raro \clh, \quad \quad (p, h)
\mapsto p(T_1, \ldots, T_n)h,\]where $p \in \mathbb{C}[\bm{z}]$ and
$h \in \clh$. We say that the module $\clh$ is a \textit{Hilbert
module} over $\mathbb{C}[\bm{z}]$ (see \cite{DP}, \cite{JS}). Denote
by $M_p: \clh \raro \clh$ the bounded linear operator
\[
M_p h = p \cdot h = p(T_1, \ldots, T_n)h, \quad \quad(h \in \clh)
\]
for $p \in \mathbb{C}[\z]$. In particular, for $p = z_i \in
\mathbb{C}[\z]$ we obtain the \textit{module multiplication}
operators as follows:
%$\{M_{z_j}\}_{j=1}^n$ by the coordinate functions
%$\{z_j\}_{j=1}^n$
\[M_{z_i} h = z_i(T_1, \ldots, T_n) h = T_i h \quad \quad(h \in
\clh,\, 1 \leq i \leq n).\] In what follows, we will use the notion
of a Hilbert module $\clh$ over $\mathbb{C}[\z]$ in place of an
$n$-tuple of commuting operators $\{T_1, \ldots, T_n\} \subseteq
\clb(\clh)$, where the operators are determined by module
multiplication by the coordinate functions, and vice versa.

\noindent A function $K: \D^n \times \D^n \rightarrow \mathbb{C}$ is
said to be \textit{positive definite kernel} (cf. \cite{Ar},
\cite{JS}) if
\[ \sum_{i, j = 1}^{k} \overline{\lambda}_i \lambda_j K(\z_i, \z_j) >
0,\] for all $\{\lambda_i\}_{i=1}^k \subseteq \mathbb{C}$,
$\{\z_i\}_{i=1}^k \subseteq \D^n$ and $k \in \mathbb{N}$. Given a
positive definite kernel $K$ on $\D^n$, the scalar-valued
reproducing kernel Hilbert space $\clh_K$ is the Hilbert space
completion of $\mbox{span}\{K(\cdot, \w) : \w \in \D^n\}$
corresponding to the inner product
\[\langle K(\cdot, \w), K(\cdot, \z)\rangle_{\clh_K} = K(\z,
\w). \quad \quad (\z, \w \in \D^n)\]The kernel function $K$ has the
reproducing property:
\[f(\w) = \langle f, K(\cdot, \w)\rangle_{\clh_K}. \quad \quad (f \in \clh_K, \w \in
\mathbb{D}^n)\]In particular, for each $\w \in \D^n$ the evaluation
operator $\bm{ev}_{\w} : \clh_K \raro \mathbb{C}$ defined by
$\bm{ev}_{\w}(f) = \langle f, K(\cdot, \w) \rangle_{\clh_K}$ ($f \in
\clh_K$) is bounded. We say that $\clh_K$ is the \textit{reproducing
kernel Hilbert space} over $\D^n$ with respect to the kernel
function $K$.

We now assume that the function $K$ is holomorphic in the first
variable and anti-holomorphic in the second variable. Then $\clh_K$
is a Hilbert space of holomorphic functions on $\D^n$ (cf.
\cite{JS}). Moreover, $\clh_K$ is said to be
\textit{reproducing kernel Hilbert module} over $\mathbb{C}[\z]$ if $1 \in
\clh_K \subseteq \clo(\D^n, \mathbb{C})$ and the module
multiplication operators $\{M_{z_i}\}_{i=1}^n$ are given by the multiplication by the
coordinate functions, that is
\[M_{z_i} f = z_i f,\]and \[(z_i f)(\w) = w_i f(\w), \quad \quad (f \in \clh_K, \w \in
\D^n)\]for $i = 1, \ldots, n$. It is easy to verify that
\[M_{z_i}^* K(\cdot, \w) = \bar{w}_i K(\cdot, \w),
\quad\quad (\w \in \D^n)\]for $i=1, \ldots, n$.

Let $\{\clh_{K_i}\}_{i=1}^n$ be a collection of reproducing kernel
Hilbert modules over $\D$ corresponding to the positive definite
kernel functions $K_i:\D \times \D \to\Comp$, $i=1,\dots,n$. Thus
 \[
  K (\z, \w)= \prod\limits_{i=1}^n K_i(z_i,w_i), \quad \quad (\z, \w
  \in \D^n)
 \]defines a positive definite kernel on $\D^n$ (cf. \cite{T}, \cite{Ar}).
Observe that $\Hil_{K_1} \ot \cdots \ot \Hil_{K_n}$ can be viewed as
a reproducing kernel Hilbert module over $\Comp[\z]$ in the
following sense:
\[\mathbb{C}[\z]\times(\Hil_{K_1}\ot\cdots\ot\Hil_{K_n}) \to
\Hil_{K_1}\ot\cdots\ot\Hil_{K_n} , \quad (p,f) \mapsto p(M_1,\ldots
, M_n)f,\]where $M_i \in B(\Hil_{K_1}\ot\cdots\ot\Hil_{K_n})$, and
\[M_i :=
 I_{\Hil_{K_1}}\otimes \cdots \otimes \underbrace{M_z}\limits_{\textup{i-th place}}
 \otimes \cdots \otimes I_{\Hil_{K_n}}.
 \quad \quad (1 \leq i \leq n)\]
Moreover, it also follows immediately from the definition of $K$
that
\begin{equation*}
\begin{split}
& \left\|\sum\limits_{i=1}^m a_i K(\cdot, \w_i)\right\|^2 =
 \sum\limits_{i=1}^m \sum\limits_{j=1}^m a_i\bar{a}_jK(\w_i,\w_j)
 = \sum\limits_{i=1}^m \sum \limits_{j=1}^m a_i\bar{a}_j
 \Big(\prod\limits_{l=1}^n K_l\big((\w_i)_l, (\w_j)_l\big)\Big)\\
 & = \sum\limits_{i=1}^m \sum\limits_{j=1}^m a_i\bar{a}_j\Big\langle
 K_1(\cdot, (\w_j)_1) \otimes \cdots \otimes
 K_n(\cdot, (\w_j)_n), K_1(\cdot, (\w_i)_1) \otimes \cdots \otimes
 K_n(\cdot, (\w_i)_n)\Big\rangle\\
 & = \left\|\sum\limits_{i=1}^m a_i K_1(\cdot, (\w_i)_1) \otimes \cdots \otimes
 K_n(\cdot, (\w_i)_n)\right\|^2,
 \end{split}
 \end{equation*}where $\{\w_i = \big((\w_i)_1, \ldots,
(\w_i)_n\big):
 1 \leq i \leq m\} \subseteq \D^n$ and
 $\{a_i\}_{i=1}^m \subseteq \mathbb{C}$ and $m \in \mathbb{N}$. Therefore, the map \[U: \mbox{span} \{K(\cdot, \w) :
 \w \in \D^n\} \longrightarrow \mbox{span}
 \{ K_1(\cdot, w_1)\otimes \cdots \otimes K_n(\cdot, w_n): \w \in \D^n\}\]
 defined by \[U K(\cdot, \w) = K_1(\cdot, w_1)\otimes \cdots \otimes K_n(\cdot,
 w_n), \quad \quad \left(\w \in \D^n\right)\]
 extends to a unitary operator from $\clh_K$ onto $\clh_{K_1}
 \otimes \cdots \otimes \clh_{K_n}$. We also have $$M_{z_i} = U^* M_i U \quad \quad (1 \leq i \leq n).$$
This implies that $\clh_K\cong \clh_{K_1}\ot\cdots\ot\clh_{K_n}$ is
a reproducing kernel Hilbert module over $\mathbb{C}[\z]$.

\textit{In what follows we identify the Hilbert tensor product of
Hilbert modules $\clh_{K_1}\ot\cdots\ot\clh_{K_n}$ with the Hilbert
module $\clh_K$ over $\mathbb{C}[\z]$. It also enables us to
identify $z^{k_1} \otimes \cdots \otimes z^{k_n}$ with $z^{\bm{k}}$
for all $\bm{k}=(k_1,\cdots,k_n) \in \mathbb{N}^n$.}

We now recall the definitions of submodules and quotient modules of
reproducing kernel Hilbert modules over $\mathbb{C}[\z]$ to be used
in this paper:

\noindent Let $\s$ and $\q$ be a pair of closed subspaces of
$\clh_K$. Then $\s$ is a \textit{submodule} of $\clh_K$ if
$M_{z_i}\s\subseteq\s$ for all $i=1,\dots,n$ and $\q$ is a
\textit{quotient module} if $\q^{\perp}(\cong \clh_K/\q)$ is a
submodule of $\clh_K$. The module multiplication operators on the
submodule $\s$ and the quotient module $\q$ of $\Hil_K$ are given by
restrictions $(R_{z_1},\dots,R_{z_n})$ and compressions
$(C_{z_1},\dots, C_{z_n})$ of the module multiplication operators
$(M_{z_1},\dots,M_{z_n})$ on $\clh_K$:
\begin{equation}
R_{z_i}:= M_{z_i}|_{\s}\quad \text{and} \quad
C_{z_i}:=P_{\q}M_{z_i}|_{\q}, \label{czi}
\end{equation}
for $i=1,\dots,n$.

\begin{defn}\label{DCD}
A quotient module $\q$ of $\Hil_K$ is \emph{doubly commuting} if for
$1\le i<j\le n$,
\[
C_{z_i}C_{z_j}^*=C_{z_j}^*C_{z_i}.
\]
A submodule $\s$ of $\Hil_K$ is \emph{doubly commuting} if for $1\le
i <j\le n$,
\[
 R_{z_i}R_{z_j}^*=R_{z_j}^*R_{z_i},
\]
and it is \emph{co-doubly commuting} if
the quotient module $\s^{\perp} (\cong \clh_K/\cls)$ is doubly commuting.
\end{defn}

The notion of a co-doubly commuting submodule was introduced in
\cite{JS1} and \cite{sarkar} in the context of Hardy module over
$\mathbb{D}^n$. However, the interplay between the doubly commuting
quotient modules and the co-doubly commuting submodules has also
been previously used by Izuchi, Nakazi and Seto and Yang \cite{III},
\cite{bshift}, \cite{Y5}, \cite{Y6}.

We end this preliminary section by recalling a result concerning
commutant of von-Neumann algebras (cf. Theorem 5.9, Chapter-IV of
~\cite{takesaki}) which will be used in later sections.

\begin{thm} \label{commutant}
Let $M$ and $N$ be two von-Neumann algebras. Then $(M\ovt N)'=M'\ovt
N'$.
\end{thm}

\section{Standard Hilbert modules}\label{sec:2}

In this section we introduce the notion of a standard reproducing
kernel Hilbert module and establish some basic properties. A
characterization of this class is also obtained which we use
throughout this note.

\begin{defn}\label{clasRKHM}
A reproducing kernel Hilbert module $\Hil \subseteq \clo(\D,
\mathbb{C})$ over $\mathbb{C}[z]$ is said to be standard Hilbert
module over $\mathbb{C}[z]$ if there does not exist two non-zero
quotient modules of $\Hil$ which are orthogonal to each other.
\end{defn}

\noindent It follows immediately that a standard Hilbert module
$\Hil$ over $\mathbb{C}[z]$ is always \textit{irreducible}, that is,
the module multiplication operator $M_z$ does not have any
non-trivial reducing subspace.

One of the pleasant features of working with a standard Hilbert
module over $\mathbb{C}[z]$ is that the quotient modules of this
space have the following useful characterization.

\begin{propn}\label{propn1}
 Let $\Hil$ be a reproducing kernel Hilbert module over $\mathbb{C}[z]$. Then
$\clh$ is a standard Hilbert module over $\mathbb{C}[z]$ if and only
if for any non-zero quotient module $\q$ of $\Hil$, the smallest
submodule containing $\q$ is $\Hil$, that is,
\[\mathop{\vee}_{l=0}^{\infty} z^l\q = \Hil.\]
\end{propn}
\begin{proof}
Let $\clh$ be a standard Hilbert module over $\mathbb{C}[z]$. Let
$\q$ be a quotient module of $\clh$ such that
\[\tilde{\q}:= \mathop{\vee} _{l=0}^{\infty} z^l\q \neq \Hil.\]
It follows that the quotient module $\tilde{\q}^{\bot}$ is
non-trivial and $\q\bot \tilde{\q}^{\bot}$. This contradicts the
assumption that $\clh$ is a standard Hilbert module.

\noindent We now turn our attention to the converse part. Let $\q_1$
and $\q_2$ be two non-zero quotient modules of $\Hil$, and
$\q_1\perp\q_2$. For all $f_1\in \q_1$ and $f_2\in \q_2$ and $l \in
\mathbb{N}$,
\[\left\langle z^lf_1,f_2\right\rangle = \left\langle M_z^lf_1,
f_2\right\rangle = \left\langle f_1, M_z^{*l}f_2\right\rangle
=0.\]This shows that \[\mathop{\vee}_{l=0}^{\infty} z^l \q_1 \bot
\q_2.\] On the other hand, $\vee_{l=0}^{\infty} z^l \q_1 = \Hil$
implies that $\q_2 = \{0\}$. This is a contradiction. Therefore,
$\clq_1$ is not orthogonal to $\clq_2$ as desired.
 \end{proof}

Our next result shows that if $K^{-1} : \mathbb{D} \times \mathbb{D}
\raro \mathbb{C}$ is a polynomial in $z$ and $\bar{w}$, then
$\clh_K$ can be realized as a standard Hilbert module over
$\mathbb{C}[z]$.

\begin{thm}\label{source}
Let $\clh_K$ be a reproducing kernel Hilbert module over
$\mathbb{C}[z]$ with reproducing kernel $K : \D\times \D\to \Comp$
such that $K^{-1}(z,w)=\sum\limits_{i,j=0}^k a_{ij} z^i\bar{w}^j$ is a
polynomial in $z$ and $\bar{w}$. Then $\Hil_K$ is a standard
Hilbert module over $\mathbb{C}[z]$.
\end{thm}

 \begin{proof}
Let $K^{-1}(z,w)=\sum\limits_{i,j=0}^k a_{ij} z^i\bar{w}^j$ and set
 $K^{-1}(M_{z}, M_{z}^{*}):=\sum\limits_{i,j=0}^k a_{ij} M_{z}^i M_{z}^{* j}$.
 For $z, w \in
\mathbb{D}$ we notice that
\[\begin{split}\langle K^{-1}(M_{z}, M_{z}^*) K(\cdot, w), K(\cdot, z)\rangle & = \sum_{i,j=0}^k \langle a_{ij} M_{z}^i
M_{z}^{* j} K(\cdot, w), K(\cdot, z) \rangle\\ & = \sum_{i,j=0}^k
\langle a_{ij} M_{z}^{* j} K(\cdot, w), M_{z}^{*i} K(\cdot, z)
\rangle \\ & = \sum_{i,j=0}^k z^i \bar{w}^j a_{ij} \langle K(\cdot,
w), K(\cdot, z) \rangle \\ & = K^{-1} (z, w) K(z, w) \\& = \langle
P_{\mathbb{C}}  K(\cdot, w), K(\cdot, z)\rangle,
\end{split}\] where $P_{\Comp}$ is the orthogonal projection of $\clh_K$ onto
the subspace of all constant functions. Consequently, it follows
that
\[ K^{-1}(M_{z}, M_{z}^*) = P_{\mathbb{C}}.\]We now assume that $\q$ is a non-zero
quotient module of $\Hil$ and $\wt{\q}=\vee_{l=0}^{\infty}z^l\q$. It
readily follows that \[P_{\Comp}(\q)=K^{-1}(M_{z},
M_{z}^*)(\q)\subseteq \wt{\q}.\] Now if $P_{\Comp}(\q)=\{0\}$, then
$\q^{\perp}$ contains the constant function $1$ and so $\q=\{0\}$
contradicting the fact that $\q\neq \{0\}$.

\noindent On the other hand, if $P_{\Comp}(\q)\neq\{0\}$, then
$1\in\wt{\q}$ and hence $\wt{\q}=\Hil$. The theorem now follows from
Proposition \ref{propn1}.
\end{proof}

\begin{rem}
We remark that the assumptions of the above theorem includes
implicitly the additional hypothesis that one can define a
functional calculus so that $\frac{1}{K}(M_z, M_z^*)$ make sense for
the kernel function $K$. It was pointed out in the paper by Arazy
and Englis \cite{AE} that for many reproducing kernel Hilbert
spaces, one can define such a $\frac{1}{K}$-calculus. In particular,
examples of standard Hilbert modules over $\mathbb{C}[z]$ includes
the weighted Bergman spaces $L^2_{a, \alpha}(\mathbb{D})$, $\alpha >
1$, with kernel functions
\[K_{a, \alpha}(z, w) = \frac{1}{(1 - z \bar{w})^{\alpha}} \quad
\quad (z, w \in \mathbb{D}).\]
\end{rem}

We will make repeated use of the following lemma concerning
commutant of $C_z = P_{\clq} M_z|_{\clq}$ on a quotient module
$\clq$ of a standard Hilbert module $\clh_K$.

\begin{lemma} \label{lemma1} Let $\Hil$
be a standard Hilbert module over $\mathbb{C}[z]$ and $\q$ be a
non-trivial quotient module of $\Hil$. Let $P$ be a non-zero
orthogonal projection in $B(\q)$. Then \[P C_z = C_z P,\]if and only
if $P = I_{\q}$.
\end{lemma}

\begin{proof}
Let $\s$ be a non-zero closed subspace of $\q$ such that \[P_{\s}
C_z = C_z P_{\s},\]or equivalently, $P_{\cls} C_z^* = C_z^*
P_{\cls}$. Hence
\[P_{\cls} M^*_z|_{\clq} = M_z^*|_{\clq} P_{\cls} = M_z^* P_{\cls}.\]
By multiplying both sides of \[P_{\cls} M^*_z|_{\clq} = M_z^*
P_{\cls},\]to the right with $P_{\s}$ we get $P_{\s}M_z^*P_{\s} =
M_z^*P_{\s}$. Hence $\s$ is a quotient module of $\clh$.

On the other hand, using $P_{\s}M_z^*P_{\s}=P_{\s}M_z^*P_{\q}$ along
with the fact that $\q$ is a quotient module we have
\begin{align*}
P_{\q \ominus \s}M_z^*P_{\q \ominus \s} &=
P_{\q}M_z^*P_{\q}-P_{\q}M_z^*P_{\s}
-P_{\s}M_z^*P_{\q}+P_{\s}M_z^*P_{\s}\\ &= M_z^*P_{\q} - M_z^*P_{\s}
= M_z^*P_{\q \ominus \s}.
\end{align*}
Thus $\q$ and $\q \ominus \s$ are two orthogonal quotient modules of
$\clh$. This contradicts the fact that $\clh$ is a standard Hilbert
module over $\mathbb{C}[z]$. Consequently, $\q \ominus \s = \{0\}$,
that is, $\clq = \s$. This completes the proof.
\end{proof}

Let $\clq$ be a quotient module of a Hilbert module $\clh$ over
$\mathbb{C}[z]$ and $\cls$ be a non-trivial closed subspace of
$\clq$. Let
\[P_{\cls} C_z = C_z P_{\cls}.\]The above proof shows that both $\s$ and
$\q\ominus\s$ are quotient modules of $\clh$. One can show that the
converse is also true. Hence this is an equivalent condition.

It is of interest to know whether an irreducible reproducing kernel
Hilbert module over $\mathbb{C}[z]$ is necessarily standard Hilbert
module over $\mathbb{C}[z]$. However, this question is not relevant
in the context of the present paper.

\section{Doubly commuting quotient module}\label{sec:3}

In this section we introduce the notion of a standard Hilbert module
in several variables. We present a characterization result for
quotient modules of standard Hilbert modules over $\mathbb{C}[\z]$,
which are doubly commuting as well as satisfy an additional natural
condition. We also obtain a characterization result for doubly
commuting quotient modules of the weighted Bergman modules over
$\mathbb{D}^n$.

We begin by defining the notion of a standard Hilbert module over
$\mathbb{C}[\z]$.
\begin{defn}
A reproducing kernel Hilbert module $\clh \subseteq \clo(\D^n,
\mathbb{C})$ over $\mathbb{C}[\z]$ is said to be a \emph{standard
Hilbert module over $\mathbb{C}[\z]$} if \[\clh = \clh_{1} \otimes
\cdots \otimes \clh_{n},\]for some standard Hilbert modules
$\{\clh_{i}\}_{i=1}^n$ over $\mathbb{C}[z]$.
\end{defn}

Here, as well as in the rest of this paper we specialize to the
class of standard Hilbert modules over $\mathbb{C}[\z]$.

The following illuminating example makes clear the connection
between the tensor product of quotient modules of standard Hilbert
modules over $\mathbb{C}[z]$ and doubly commuting quotient modules
of standard Hilbert modules over $\mathbb{C}[\z]$:

\noindent Let $\clh = \clh_1 \otimes \cdots \otimes \clh_n$ be a
standard Hilbert module over $\mathbb{C}[\z]$,
and let $\q_j \subseteq \clh_j$ be a quotient module for each $j =
1, \ldots, n$. Then
\[
\q=\q_1\ot\cdots\ot\q_n,
 \]is a doubly commuting quotient module of $\Hil_1\ot\cdots\ot\Hil_n$
 with the module multiplication operators
\[
I_{\q_1}\ot \cdots\ot \underbrace{P_{\q_i}M_{z}|_{\q_i}}
\limits_{i\text{-th}}\ot\cdots\ot I_{\q_n}\quad (i=1,\dots,n).
 \]

The purpose of this section is to prove that under a rather natural
condition a doubly commuting quotient module of a
standard Hilbert module over $\mathbb{C}[\z]$ is
always represented in the above form.

The key ingredient in our approach will be the following
propositions concerning reducing subspaces of standard Hilbert modules.

\begin{propn}\label{ppmain1}
Let $\clh = \Hil_1\ot\cdots\ot\Hil_n$ be a standard Hilbert module
over $\mathbb{C}[\z]$. Let $\q$ be a closed subspace of $\clh$ and
let $k \in \{1, \ldots, n\}$. Then $\q$ is $M_{z_i}$-reducing for
$i=k,k+1,\ldots ,n$, if and only if \[\q= \mathcal{E} \ot
\Hil_k\ot\cdots\ot\Hil_n,\]for some closed subspace
$\mathcal{E}\subseteq \Hil_1\ot\cdots\ot\Hil_{k-1}.$
\end{propn}

\begin{proof}
For $k\le i\le n$, let $\mathcal{N}_i$ be the von-Neumann algebra
generated by $\{I_{\Hil_i},M_{z}\}$, where $M_z$ is the module
multiplication operator on $\Hil_i$. It follows immediately that the
von-Neumann algebra generated by \[\{I_{\clh}, M_{z_i}:
i=k,k+1,\ldots ,n\} \subseteq \clb(\clh_1 \otimes \cdots \otimes
\clh_n),\]is given by
$$\mathbb{C}I_{\Hil_1\ot\cdots\ot\Hil_{k-1}}\ovt\mathcal{N}_k\ovt\cdots\ovt
\mathcal{N}_n.$$ By virtue of Lemma \ref{lemma1} we
have\[\mathcal{N}^{\prime}_i=\Comp I_{\Hil_i}. \quad \quad (k\le
i\le n)\] On account of Theorem \ref{commutant} we have then
\[\left(\mathbb{C}I_{\Hil_1\ot\cdots\ot\Hil_{k-1}}\ovt\mathcal{N}_k\ovt\cdots\ovt
\mathcal{N}_n\right)^{\prime} =
\clb(\Hil_1\ot\cdots\ot\Hil_{k-1})\ovt
\mathbb{C}I_{\Hil_k\ot\cdots\ot\Hil_n},\]and hence $\q$ is
$M_{z_i}$-reducing subspace for all $i=k,k+1,\ldots ,n$, if and only
if \[P_{\q} \in \left(\mathbb{C}I_{\Hil_1\ot\cdots\ot\Hil_{k-1}}\ovt
\mathcal{N}_k\ovt\cdots\ovt \mathcal{N}_n\right)^{\prime} =
\clb(\Hil_1\ot\cdots\ot\Hil_{k-1})\ovt
\mathbb{C}I_{\Hil_k\ot\cdots\ot\Hil_n}.\] On the other hand, since
$P_{\clq}$ is a projection in
$\clb(\Hil_1\ot\cdots\ot\Hil_{k-1})\ovt
\mathbb{C}I_{\Hil_k\ot\cdots\ot\Hil_n}$, there exists a closed
subspace $\mathcal{E}$ of $\Hil_1\ot\cdots\ot\Hil_{k-1}$ such that
\[P_{\clq} = P_{\mathcal{E}}\ot I_{\Hil_k\ot\cdots\ot\Hil_n}.\]
Hence it follows that \[\q = \mathcal{E}\ot
 \Hil_k\ot\cdots\ot\Hil_n.\]This completes the proof.
 \end{proof}

\begin{propn}\label{ppmain2}
Let $\Hil=\Hil_1\ot\cdots\ot\Hil_n$ be a standard Hilbert module
over $\mathbb{C}[\z]$ and let $\q_1$ be a quotient module of
$\Hil_1.$ Then a closed subspace $\mathcal{M}$ of $\q : = \q_1
\otimes \Hil_2\ot\cdots\ot\Hil_n$ is $P_{\q}M_{z_1}|_{\q}$-reducing
if and only if there exists a closed subspace $\mathcal{E}$ of
$\Hil_2\ot\cdots\ot\Hil_n$ such that \[\mathcal{M} = \q_1 \otimes
\mathcal{E}.\]
\end{propn}

\begin{proof} Suppose $\clq_1$ is a quotient module of $\clh_1$.
We observe that \[P_{\q}M_{z_1}|_{\q} =
\left(P_{\q_1}M_z|_{\q_1}\otimes
I_{\Hil_2\ot\cdots\ot\Hil_n}\right).\] We also note that a closed
subspace $\mathcal{M}$ of $\q$ is $P_{\q}M_{z_1}|_{\q}$-reducing if
and only if \[P_{\mathcal{M}}\in \left(\mathcal{N}\ovt
 I_{\Hil_2\ot\cdots\ot\Hil_n}\right)^{\prime},\]
where $\mathcal{N} \subseteq \clb(\q_1)$ is the von-Neumann algebra
generated by $\{I_{\q_1}, P_{\q_1}M_z|_{\q_1}\}$. Now
\[\left(\mathcal{N}\ovt I_{\Hil_2\ot\cdots\ot\Hil_n}\right)^{\prime} =
\mathcal{N}^{\prime} \ovt \clb(\Hil_2\ot\cdots\ot\Hil_n).\]By Lemma
\ref{lemma1} we have $\mathcal{N}'=\Comp I_{\q_1}$ and hence
 \[\left(\mathcal{N}\ovt I_{\Hil_2\ot\cdots\ot\Hil_n}\right)' =
 \mathbb{C}I_{\q_1}~\bar{\otimes} \clb(\Hil_2\ot\cdots\ot\Hil_n).\]Therefore,
 $P_{\mathcal{M}}\in \left(\mathcal{N}\ovt
 I_{\Hil_2\ot\cdots\ot\Hil_n}\right)'$ if and only if \[P_{\mathcal{M}}= I_{\q_1} \otimes
 P_{\mathcal{E}},\]that is, $\mathcal{M}=\q_1\ot
 \mathcal{E}$, for some closed subspace $\mathcal{E}$ of
 $\Hil_2\ot\cdots\ot\Hil_n$.

\noindent The sufficiency part is trivial. This completes the proof.
\end{proof}

Let $\q$ be a quotient module of a standard Hilbert module over
$\mathbb{C}[\z]$. For $1\le k\le n$, let $[\q]_{z_{k},z_{k+1},\ldots
,z_{n}}$ denote the smallest joint
$(M_{z_k},\dots,M_{z_n})$-invariant subspace containing $\q$. That
is,
\begin{equation}\label{generating}
[\q]_{z_{k},z_{k+1},\ldots ,z_{n}}:=
\bigvee\limits_{(l_k,l_{k+1},\dots,l_n)\in \Nat^{(n-k+1)}}
M_{z_k}^{l_k}\cdot M_{z_{k+1}}^{l_{k+1}}\cdot \cdot \cdot
M_{z_{n}}^{l_n}\q.
 \end{equation}

We are now ready to prove the characterization result concerning
tensor product of quotient modules of standard Hilbert modules over
$\mathbb{C}[\z]$.

\begin{thm} \label{main theorem}
Let $\q$ be a quotient module of a standard Hilbert module $\clh =
\Hil_1\ot\cdots\ot\Hil_n$ over $\mathbb{C}[\z]$. Then $$\q =
\q_1\otimes \cdots \otimes \q_n,$$ for some quotient
module $\q_i$ of $\Hil_i$, $i=1,\dots,n$, if and only if\\
\textup{(i)} $\q$ is doubly commuting, and\\
\textup{(ii)} $[\q]_{z_{k},z_{k+1},\ldots ,z_{n}}$ is a joint
$\left(M_{z_{k}},M_{z_{k+1}},\ldots ,
 M_{z_{n}}\right)$-reducing subspace of $\Hil_1\ot\cdots\ot\Hil_n$
 for $k = 1, \ldots, n$.
\end{thm}

\begin{proof}
Let $\q$ be a doubly commuting quotient module of $\clh$ and
$[\q]_{z_{k},z_{k+1},\ldots ,z_{n}}$ be a joint
$\left(M_{z_{k}},M_{z_{k+1}},\ldots, M_{z_{n}}\right)$-reducing
subspace for $k = 1, \ldots, n$. In particular for $k = 2$,
\[\wt{\q}:=[\q]_{z_{2},z_{3},\ldots ,z_{n}},\] is a joint
$\left(M_{z_{2}},M_{z_{3}},\ldots , M_{z_{n}}\right)$-reducing
subspace of $\Hil_1\ot\cdots\ot\Hil_n$. By virtue of
Proposition~\ref{ppmain1} we have $$\wt{\q} = \q_1 \otimes
\Hil_2\ot\cdots\ot\Hil_n,$$ for some closed subspace $\q_1$ of
$\Hil_1$. Also since $\q$ is a quotient module and $M_{z_i}^*$
commutes with $M_{z_j}$ for $i\neq j$, it follows that $\wt{\q}$ is
a $M_{z_1}^*$-invariant subspace. Hence $\q_1$ is a quotient module
of $\Hil_1$. Now we claim that $\q$ is a
$P_{\wt{\q}}M_{z_1}|_{\wt{\q}}$-reducing subspace of $\wt{\q}$. To
this end, since $\q\subseteq \wt{\q}$, it is enough to show that
$$ P_{\q}M_{z_1}^*|_{\wt{\q}} = M_{z_1}^*|_{\q}.$$
Using the fact that $\q$ is doubly commuting it follows that
\[C_{z_1}^*C_{z_i}^l = C_{z_i}^lC_{z_1}^*,\]for
$l\geq 0$ and $2\leq i\leq n$, and hence
\[C_{z_1}^*C_{z_2}^{l_2}\cdots C_{z_n}^{l_n} = C_{z_2}^{l_2}\cdots
C_{z_n}^{l_n}C_{z_1}^*,\] for $l_2, l_3, \ldots ,l_n\geq 0$.
Therefore
\[ M_{z_1}^*P_{\q}M_{z_2}^{l_2}\cdots M_{z_n}^{l_n}P_{\q} =P_{\q}M_{z_2}^{l_2}\cdots
M_{z_n}^{l_n}M_{z_1}^*P_{\q}. \quad \quad (l_2, l_3, \ldots ,l_n\geq
0)\]This implies that \[M_{z_1}^*P_{\q}(M_{z_2}^{l_2}\cdots
M_{z_n}^{l_n}P_{\q})= P_{\q}M_{z_1}^*(M_{z_2}^{l_2}\cdots
M_{z_n}^{l_n}P_{\q}),\] for $l_2, l_3, \ldots ,l_n\geq 0$. This
proves the claim.

\noindent Now applying Proposition \ref{ppmain2}, we obtain a closed
subspace $\mathcal{E}_1$ of $\Hil_2\ot\cdots\ot\Hil_n$ such that
\[\q = \q_1 \otimes \mathcal{E}_1.\]Finally note that since $\q$ is doubly
commuting, $\mathcal{E}_1$ is also doubly commuting quotient module
of $\Hil_2\ot\cdots\ot\Hil_n$ and it satisfies the condition (ii) in
the statement of this theorem. Repeating the argument above for
$\mathcal{E}_1$, we conclude that \[\mathcal{E}_1=\q_2\ot
\mathcal{E}_2,\]for some quotient module $\q_2$ of $\Hil_2$ and
doubly commuting quotient module $\mathcal{E}_2$ of
$\Hil_3\ot\cdots\ot\Hil_n$. Continuing in this way we obtain
quotient modules $\q_i \subseteq \Hil_i$, for $i=1,\dots,n$, such
that $$\q = \q_1 \otimes \q_2\otimes \cdots \otimes \q_n.$$This
proves the sufficient part.

\noindent To prove the necessary part, let $\q = \q_1 \otimes
\q_2\otimes \cdots \otimes \q_n$ be a quotient module of $\Hil$.
Clearly
$$(I_{\q_1}\otimes I_{\q_2}\otimes \cdots \otimes
\underbrace{P_{\q_i}M_z|_{\q_i}}\limits_{i-\textup{th place}}\otimes
\cdots \otimes I_{\q_n})_{i=1}^n,$$is a doubly commuting tuple, that
is, $\q$ is doubly commuting. Finally, using the fact that $\Hil_i$
is a standard Hilbert module over $\mathbb{C}[z]$
for all $i=1,\dots,n$, we have
\begin{align*} [\q]_{z_k,\ldots,z_n}&=\q_1\ot\cdots\ot \q_{k-1}\ot
[\q_k]_{z}\ot\cdots\ot [\q_n]_{z}\\&=
\q_1\ot\cdots\ot\q_{k-1}\ot\Hil_k\ot\cdots\ot\Hil_n,
\end{align*}for $1\le k\le n$. This and Proposition \ref{ppmain1} proves (ii).
This completes the proof.
\end{proof}

\begin{rem}

Let $\Hil_i$ be a Hilbert module over $\mathbb{C}[z]$ with module
multiplication operator $T_i$, $i = 1, \ldots, n$. Moreover, assume
that $\clh_i$ is a standard Hilbert module over $\mathbb{C}[z]$,
that is, there does not exists a pair of non-zero quotient modules
$\clq_1$ and $\clq_2$ such that $\clq_1 \perp \clq_2$. In this case,
the above theorem still remains true for the Hilbert module
$\Hil=\Hil_1\ot\cdots\ot\Hil_n$ over $\mathbb{C}[\z]$ with module
multiplication operators
\[
\{I_{\Hil_1}\ot\cdots\ot I_{\Hil_{i-1}}\ot T_i\ot I_{\Hil_{i+1}}\ot
\cdots\ot I_{\Hil_n}\}_{i=1}^n.
\]
\end{rem}

Let $\Hil_i$ be a reproducing kernel Hilbert module over
$\mathbb{C}[z]$ with kernel $K_i$ such that $K_i^{-1}$ is a
polynomial for all $i=1,\dots,n$. Then by  Theorem \ref{source} we
know that $\Hil_i$'s are standard Hilbert modules
over $\mathbb{C}[z]$  (see also the remark following Theorem
\ref{source}). Thus $\Hil=\Hil_1\ot\cdots\ot\Hil_n$ is a standard
Hilbert module over $\mathbb{C}[\z]$. This
subclass of standard Hilbert modules over
$\mathbb{C}[\z]$ plays the central role in the rest of this paper.
So we make the following definition to refer this subclass.

\begin{defn}\label{anal-defn}
A standard Hilbert module $\clh = \Hil_{K_1} \ot \cdots \ot
\Hil_{K_n}$ over $\mathbb{C}[\z]$ is said to be \textit{analytic
Hilbert module} if $K_i^{-1}$ is a polynomial in two variables $z$
and $\bar{w}$ for all $i=1,\dots,n$.
\end{defn}

The notion of analytic Hilbert module is closely related to the
$\frac{1}{K}$-calculus introduced by Arazy and Englis \cite{AE}. Our
result is true in the generality of Arazy-Englis. However, to avoid
technical complications we restrict our attention to the analytic
Hilbert modules.
%instead of $\frac{1}{K}$-calculus.

Let $\clh$ be a standard Hilbert module over $\mathbb{C}[\z]$. Then
$\clh$ is an analytic Hilbert module if and only if  $K^{-1}(\z,
\w)$ is a polynomial in $z_1, \ldots, z_n, \bar{w}_1, \ldots,
\bar{w}_n$.

We show now that the condition (ii) in Theorem~\ref{main theorem}
holds for any quotient module of an analytic Hilbert module over
$\mathbb{C}[\z]$. After the proof of the proposition we will give
some examples in order.

\begin{propn}
Let $\q$ be a non-zero quotient module of an analytic module
$\Hil=\Hil_1\ot\cdots\ot\Hil_n$ over $\mathbb{C}[\z]$. Then
$[\q]_{z_k,\dots,z_n}$ is
$(M_{z_k},M_{z_{k+1}},\dots,M_{z_n})$-reducing subspace for $k = 1,
\ldots, n$.
\end{propn}

\begin{proof}
Let $1\le k\le n$ be fixed. Set \[\prod_{i=k}^n K_i^{-1}(z_i,{w}_i)=
\sum_{\mathbf{l},\mathbf{m}\in \Nat^{(n-k+1)}}
a_{\mathbf{l},\mathbf{m}} z^{\bm{l}}\bar{w}^{\bm{m}},\] where
$z^{\bm{l}}=z_k^{l_k}\cdots z_n^{l_n}$ and $\bar{w}^{\bm{m}} =
\bar{w}_k^{m_k} \cdots \bar{w}_n^{m_n}$ and
$\bm{l}=(l_k,\dots,l_{n})$ and $\bm{m}=(m_k,\dots,m_{n})$ are in
$\Nat^{(n-k+1)}$. Likewise, if $\mathbf{l}=(l_k,\dots,l_{n}) \in
\Nat^{(n-k+1)}$, then define $M_{z}^{\mathbf{l}}=M_{z_k}^{l_k}\cdots
M_{z_{n}}^{l_{n}}$. Notice first that
\begin{equation}
\label{projection} I_{\Hil_1\ot\cdots\ot\Hil_{k-1}}\ot
P_{\Comp}^{\ot (n-k+1)}= \prod\limits_{i=k}^{n}
K_i^{-1}(M_{z_i},M_{z_i}^*)= \sum\limits_{\mathbf{l},
\mathbf{m}\in\Nat^{(n-k+1)}}a_{\mathbf{l},\mathbf{m}}
 M_{z}^{\mathbf{l}}M_{z}^{* \mathbf{m}}.
\end{equation}
In the last equality we used the fact that $M_{z_i} M_{z_j}^* =
M_{z_j}^* M_{z_i}$ for $i\neq j$. This implies
\[\big(I_{\Hil_1\ot\cdots\ot\Hil_{k-1}}\ot P_{\Comp^{\ot
(n-k+1)}}\big)(\q) \subseteq [\q]_{z_k,\dots,z_n}.
\]
By a similar argument as the in the proof of Theorem \ref{source},
we have \[\big(I_{\Hil_1\ot\cdots\ot\Hil_{k-1}}\ot P_{\Comp^{\ot
(n-k+1)}}\big)(\q)\neq \{0\}.\]Setting
\[\big(I_{\Hil_1\ot\cdots\ot\Hil_{k-1}}\ot P_{\Comp^{\ot
(n-k+1)}}\big)(\q) = \q_1\ot \Comp^{\ot (n-k+1)},\] for some closed
subspace $\q_1$ of $\Hil_1\ot\cdots\ot\Hil_{k-1}$, we obtain
\[\q_1\ot \Hil_k\ot\cdots\ot\Hil_n\subseteq [\q]_{z_k,\dots,z_n}.\]To see
$ [\q]_{z_k,\dots,z_n}\subseteq \q_1\ot \Hil_k\ot\cdots\ot\Hil_n$,
it is enough to prove that $\q\subseteq \q_1\ot
\Hil_k\ot\cdots\ot\Hil_n$, or equivalently,\[\q_1^{\perp}\ot
\Hil_k\ot\cdots\ot\Hil_n\subseteq \q^{\perp}.\]Since $\q^{\perp}$
 is a submodule the last containment will follow if we show that
 $f\ot \underbrace{1\ot\cdots\ot 1}\limits_{(n-k+1)-\text{times}}
 \in \q^{\perp}$ for any $f\in \q_1^{\perp}$.
 Now for $f\in \q_1^{\perp}$ and $g\in\q$, we have
\begin{align*}
\langle f\ot 1\ot\cdots\ot 1,g\rangle&= \langle (
I_{\Hil_1\ot\cdots\ot\Hil_{k-1}}\ot P_{\Comp^{\ot (n-k+1)}}) (f\ot
1\ot\cdots\ot 1),g\rangle\\&=\langle f\ot 1\ot\cdots\ot 1,
(I_{\Hil_1\ot\cdots\ot\Hil_{k-1}}\ot P_{\Comp^{\ot
(n-k+1)}})g\rangle\\&=0,
\end{align*}
where the last equality follows from the fact that $
(I_{\Hil_{1}\ot\cdots\ot\Hil_{k-1}}\ot P_{\Comp^{\ot (n-k+1)}})g\in
\q_1\ot \Comp^{\ot (n-k+1)}$. Therefore for any $1\le k\le n$,
\[[\q]_{z_k,\dots,z_n}=\q_1\ot\Hil_k\ot\cdots\ot\Hil_n,\]
for some closed subspace $\q_1$ of $\Hil_1\ot\cdots\ot\Hil_{k-1}$.
The result now follows from Proposition~\ref{ppmain1}.
\end{proof}

Combining above proposition, Theorems \ref{source} and \ref{main
theorem} we have the following result.

\begin{thm}\label{analytic module}
Let $\clq$ be a quotient module of an analytic Hilbert module
$\Hil=\Hil_1\ot\cdots\ot\Hil_n$ over $\mathbb{C}[\z]$. Then the
following conditions are equivalent:

\textup{(i)} $\q$ is doubly commuting.

\textup{(ii)} $\q=\q_1\ot\cdots\ot\q_n$ for some quotient module
$\q_i$ of $\Hil_i$, $i=1,\dots,n$.
\end{thm}

Now we pass to discuss some examples of analytic Hilbert modules and
applications of Theorem \ref{analytic module}. First consider the
case of the Hardy module $H^2(\D^n)$ over the unit polydisc $\D^n$.
The kernel function of $H^2(\D)$ is given by
\[\mathbb{S}(z,w)=\frac{1}{1-z\bar{w}}. \quad \quad (z , w \in \D)\]
In particular, $\mathbb{S}^{-1}(z, w)$ is a polynomial. On account
of the Hilbert module isomorphism
$$H^2(\D^n)\cong \underbrace{H^2(\D)\ot\cdots\ot H^2(\D)}
\limits_{n\text{-times}},$$ we recover the following result of
\cite{sarkar} (Theorem 3.2) and \cite{III}.

\begin{thm}\label{sarkar}
Let $\q$ be a quotient module of $H^2(\D^n)$. Then $\q$ is doubly
commuting if and only if $\q=\q_1\ot\cdots\ot\q_n$ for some quotient
modules $\q_1,\dots,\q_n$ of $H^2(\D)$.
\end{thm}

Next we consider the case of \textit{weighted Bergman spaces} over
$\mathbb{D}^n$. The weighted Bergman spaces over the unit disc is
denoted by $L^2_{a, \alpha}(\D)$, with $\alpha > -1$, and is defined
by
$$L^2_{a, \alpha}(\D):=\{f\in \mathcal{O}(\D): \int_{\D}\vert f(z)\vert^2\
dA_{\alpha}(z)<\infty\},$$ where $dA_{\alpha}(z)=(a+1)(1-\vert
z\vert^2)^{a} dA(z)$ and $dA$ refers the normalized area measure on
$\D$. The weighted Bergman modules are reproducing kernel Hilbert
modules with kernel functions
\[K_{\alpha}(z,w)=\frac{1}{(1-z\bar{w})^{\alpha+2}}. \quad \quad (z, w \in \D)\]
It is evident that $K_{\alpha}^{-1}$ is a polynomial if $\alpha
\in\Nat$. Let $\bm{\alpha} \in \mathbb{Z}^n$ with $\alpha_i > -1$
for $i = 1, \ldots, n$. The weighted Bergman space $L^2_{a,
\bm{\alpha}}(\D^n)$ over $\D^n$ with weight $\bm{\alpha}$ is a
standard Hilbert module over $\mathbb{C}[\z]$
with kernel function
\[K_{\bm{\alpha}}(\z,{\w}):=\prod_{i=1}^nK_{\alpha_i}(z_i,{w_i})
= \prod_{i=1}^n \frac{1}{(1-z_i \bar{w}_i)^{\alpha_i+2}}. \quad
\quad (\z, \w \in \D^n)\]Thus we have the following theorem.

\begin{thm}\label{main result}
Let $\bm{\alpha}=(\alpha_1,\dots, \alpha_n)\in \mathbb{Z}^n$ with
$\alpha_i > -1$ for $i = 1, \ldots, n$. Then a quotient module $\q$
of $L^2_{a, \bm{\alpha}}(\D^n)$ is doubly commuting if and only if
$\q=\q_1\ot \cdots\ot\q_n$ for some quotient modules $\q_i$ of
$L^2_{a, \alpha_i}(\D)$, $i=1,\dots,n$.
 \end{thm}

Note that by the remark after Theorem~\ref{source} the above
characterization result also holds for
$\bm{\alpha}=(\alpha_1,\dots,\alpha_n)\in \mathbb{R}^n$ with
$\alpha_i>1$, $i=1,\dots,n$.

\section{Co-doubly commuting submodules}\label{sec:4}

The purpose of this section is twofold. First, we explicitly compute
the cross commutators of a co-doubly commuting submodule (see
Definition \ref{DCD}) of analytic Hilbert modules over
$\mathbb{C}[\z]$. Second, we investigate a variety of issues related
to essential doubly commutativity of co-doubly commuting submodules.
In particular, we completely classify the class of co-doubly
commuting submodules which are essentially doubly commuting for
$n\ge 3$.

We start with a well known result (cf. \cite{sarkar}) concerning sum
of a family of commuting orthogonal projections on Hilbert spaces.

\begin{lemma}\label{P-F} Let $\{P_i\}_{i=1}^n$ be
a collection of commuting orthogonal
projections on a Hilbert space $\clh$. Then $\cll :=
\mathop{\sum}_{i=1}^n \mbox{ran} P_i$ is closed and the orthogonal
projection of $\clh$ onto $\cll$ is given by
\[P_{\cll}
= I_{\clh} - \mathop{\prod}_{i=1}^n (I_{\clh} - P_i).\]
\end{lemma}

Now we are ready to present a characterization of co-doubly
commuting submodules of an analytic Hilbert module $\mathbb{C}[\z]$.
Recall that a submodule $\s$ of an analytic Hilbert module $\Hil$
over $\mathbb{C}[\z]$ is co-doubly commuting if $\q=\s^{\perp}(\cong
\Hil/\s)$ is doubly commuting.

\begin{thm} \label{co-doubly commuting}
Let $\Hil=\Hil_1\ot\cdots\ot\Hil_n$ be an analytic Hilbert module
over $\mathbb{C}[\z]$ and $\s$ be a submodule of $\clh$. Then $\cls$
is co-doubly commuting if and only if
\[\s=(\q_1\ot\cdots\ot\q_n)^{\perp}= \sum_{i=1}^n
\Hil_1\ot\cdots\ot\Hil_{i-1}\ot\q_i^{\perp}\ot\Hil_{i+1}\ot\cdots\ot\Hil_n,\]
for some quotient module $\q_i$ of $\Hil_i$ and $i=1,\dots,n$.
\end{thm}

\begin{proof}
Let $\cls$ be a co-doubly commuting submodule of $\clh$. Applying
Theorem~\ref{analytic module} to $\cls$ we have
\[\s=(\q_1\ot\cdots\ot\q_n)^{\perp},\]
for some quotient module $\q_i$ of $\Hil_i$ and $i=1,\dots,n$. Now
let $P_i$ be the orthogonal projection of $\clh$ onto
$\Hil_1\ot\cdots\ot\Hil_{i-1}\ot\q_i^{\perp}\ot\Hil_{i+1}\ot\cdots\ot\Hil_n$.
Then $\{P_i\}_{i=1}^n$ satisfies the hypothesis of Lemma \ref{P-F}.
Also note that $\clq_1 \otimes \cdots \otimes \clq_n$ is the range
of the orthogonal projection of $\prod_{i=1}^n (I_{\clh} - P_i)$,
that is,\[P_{\clq_1 \otimes \cdots \otimes \clq_n} = \prod_{i=1}^n
(I_{\clh} - P_i).\] From this and Lemma \ref{P-F} we readily obtain
\[\cls = \sum_{i=1}^n
\Hil_1\ot\cdots\ot\Hil_{i-1}\ot\q_i^{\perp}\ot\Hil_{i+1}\ot\cdots\ot\Hil_n.\]This
completes the proof.
\end{proof}

In the sequel we will make use of the following notation.

\noindent Let $\q=\q_1\ot\cdots\ot\q_n$ be a doubly commuting
quotient module of an analytic Hilbert module $\clh = \clh_1 \otimes
\cdots \otimes \clh_n$ over $\mathbb{C}[\z]$ , where $\clq_i$ is a
quotient module of $\clh_i$, $i = 1, \ldots, n$. Let $\bm{\lambda}
=\{\lambda_1,\dots,\lambda_k\}$ be a non-empty subset of
$\{1,\dots,n\}$. The subspace $\q_{\bm{\lambda}}^{\perp}$ of $\clh$
is defined by
\begin{equation}\q_{\bm{\lambda}}^{\perp}:=\q_1\ot\cdots\ot\underbrace{\q_{\lambda_1}^{\perp}}
\limits_{\lambda_1\text{-th}}\ot\cdots\ot
\underbrace{\q_{\lambda_k}^{\perp}}
\limits_{\lambda_k\text{-th}}\ot\cdots\ot\q_n. \label{qperp}
\end{equation}
%In particular, for $\bm{\lambda} =\{k\}$, $1\le k\le n$, we simply write $\q_{[k]}^{\perp}$
%instead of $\q_{[\Lambda]}^{\perp}$ and in this case
%$\q_{[k]}^{\perp}=
%\q_1\ot\cdots\ot\q_{k-1}\ot \q_k^{\perp}\ot\q_{k+1}\ot\cdots\ot\q_n$.
Notice that \[\q_{\bm{\lambda}}^{\perp} \perp
\q_{\bm{\lambda}'}^{\perp},\]for each non-empty $\bm{\lambda},
\bm{\lambda}' \subseteq \{1, \ldots, n\}$ and $\bm{\lambda} \neq
\bm{\lambda}'$. This implies that \[(\q_1\ot\cdots\ot\q_n)^{\perp}=
\bigoplus_{\emptyset\neq\bm{\lambda} \subseteq \{1,\dots,n\}}
\q_{\bm{\lambda}}^{\perp}.\]

The following theorem provides us with an easy way to calculate the
cross commutators of co-doubly commuting submodules of analytic
Hilbert modules over $\mathbb{C}[\z]$.

\begin{thm}
Let $\Hil=\Hil_1\ot\cdots\ot\Hil_n$ be an analytic Hilbert module
over $\mathbb{C}[\z]$ and $\s= (\q_1\ot\cdots\ot\q_n)^{\perp}$ be a
co-doubly commuting submodule of $\clh$. Then for all $1\leq i<j\leq
n,$
\[[R^*_{z_i},R_{z_j}] = P_{\q_1}\ot\cdots\ot
\underbrace{P_{\q_i}M_z^*P_{\q_i^{\perp}}}\limits_{i\textup{-th}}\ot
\cdots \ot \underbrace{P_{\q_j^{\perp}}M_zP_{\q_j}}
\limits_{j\textup{-th}}\ot\cdots \ot P_{\q_n},\]where $R_{z_j}=
M_{z_j}|_{\s} $ for $1\le j\le n$.
 \end{thm}

\begin{proof} Let $\s= (\q_1\ot\cdots\ot\q_n)^{\perp}$ be a
co-doubly commuting submodule of $\clh$. By definition $R_{z_l} =
M_{z_l}|_{\cls}$ and hence $R_{z_l}^* = P_{\cls} M_{z_l}^*|_{\cls}$
for $l = 1, \ldots, n$. Let  $1\le i<j\le n$.
Then
\begin{align*}[R^*_{z_i},R_{z_j}]&= R^*_{z_i} R_{z_j} - R_{z_j} R^*_{z_i}\\
& = P_{\s}M_{z_i}^*M_{z_j}|_{\s}-
P_{\s} M_{z_{j}}P_{\s}M_{z_i}^*|_{\s}\\
&=P_{\s}M_{z_i}^*M_{z_j}|_{\s}-P_{\s}M_{z_j}(I-P_{\s^{\perp}})M_{z_i}^*|_{\s}\\
&=P_{\s}M_{z_j}P_{\s^{\perp}}M_{z_i}^*|_{\s}\\
&= P_{\cls} M_{z_j}P_{\q_1\ot\cdots\ot\q_n}M_{z_i}^* P_{\cls}.
\end{align*}
Combining this with (\ref{qperp}), we have
\[[R^*_{z_i},R_{z_j}] = \Big(\sum_{\emptyset\neq \bm{\lambda} \subseteq\{1,\dots,n\}}
P_{\q_{\bm{\lambda}}^{\perp}}\Big)
M_{z_j}P_{\q_1\ot\cdots\ot\q_n}M_{z_i}^*\Big(\sum_{\emptyset\neq
\bm{\lambda}'\subseteq\{1,\dots,n\}}P_{\q_{\bm{\lambda}'}^{\perp}}\Big).\]
Observe that for each $\bm{\lambda} \neq \{l\}$ and $l \in \{1,
\ldots, n\}$,
\[P_{\q_1\ot\cdots\ot\q_n}M_{z_l}^*P_{\q_{\bm{\lambda}}^{\perp}}=0,\]
and therefore
\[\begin{split}[R^*_{z_i},R_{z_j}]& = \Big(\sum_{\emptyset\neq \bm{\lambda}
\subseteq\{1,\dots,n\}}
P_{\q_{\bm{\lambda}}^{\perp}}\Big)
M_{z_j}P_{\q_1\ot\cdots\ot\q_n}M_{z_i}^*\Big(\sum_{\emptyset\neq
\bm{\lambda}'\subseteq\{1,\dots,n\}}P_{\q_{\bm{\lambda}'}^{\perp}}\Big)\\
& = \sum_{\emptyset\neq \bm{\lambda}, \bm{\lambda}'
\subseteq\{1,\dots,n\}} P_{\q_{\bm{\lambda}}^{\perp}}
M_{z_j}P_{\q_1\ot\cdots\ot\q_n}M_{z_i}^*
P_{\q_{\bm{\lambda}'}^{\perp}} \\ & = P_{\q_{\{j\}}^{\perp}}M_{z_j}
P_{\q_1\ot\cdots\ot\q_n}M_{z_i}^*P_{\q_{\{i\}}^{\perp}}\\
&= P_{\q_1}\ot\cdots\ot \underbrace{\big(P_{\q_i}M_z^*
P_{\q_i^{\perp}}\big)}\limits_{i\textup{-th}} \ot \cdots \ot
\underbrace{\big(P_{\q_j^{\perp}}M_zP_{\q_j}\big)}
\limits_{j\textup{-th}}\ot\cdots \ot P_{\q_n}.
\end{split}\]
This completes the proof.
\end{proof}

We still need a few more definitions about "small commutators" on
Hilbert spaces.

Let $\clh$ be a Hilbert module over $\mathbb{C}[\z]$. Let $\cls$ and
$\clq$ be submodule and quotient module of $\clh$, respectively.
Then $\s$ is said to be \emph{essentially doubly commuting} if
\[[R_{z_i}^*, R_{z_j}]\in \mathcal{K}(\s),\] for $1\le i<j\le n$.
Here $\mathcal{K}(\s)$ denotes the algebra of all compact operators
on $\cls$. Moreover, it is \emph{essentially normal} if $[R_{z_i}^*,
R_{z_j}]\in \mathcal{K}(\s)$ for $1\le i,j\le n$. Similarly a
quotient module $\q$ of a Hilbert module
$\Hil=\Hil_1\ot\cdots\ot\Hil_n$ is \emph{essentially doubly
commuting} if \[[C_{z_i}^*, C_{z_j}]\in \mathcal{K}(\q),\]for all
$1\le i<j\le n$ and it is \emph{essentially normal} if $[C_{z_i}^*,
C_{z_j}]\in \mathcal{K}(\q)$ for $1\le i,j\le n$ (see \cite{JS1}).
Here $R_{z_i}$ and $C_{z_i}$ are as in \eqref{czi}.

Now we can give a characterization of essentially doubly commuting
co-doubly commuting submodules of analytic Hilbert modules over
$\mathbb{C}[\z]$.

\begin{thm}\label{edc} Let $\s=(\q_1\ot\cdots\ot\q_n)^{\perp}$ be a
co-doubly commuting submodule of an analytic Hilbert module
$\Hil=\Hil_1\ot\cdots\ot\Hil_n$ over $\mathbb{C}[\z]$, where $\q_i$
is a quotient module of $\Hil_i$, $i=1,\dots,n$. Then:

\noindent \textup{(i)} For $n=2$, $\s$ is essentially doubly
commuting if and only if $P_{\q_j}M_z^*P_{\q_j^{\perp}}$ is compact
for all
$j=1,2$.\\
\textup{(ii)} For $n>2$, $\s$ is essentially doubly
commuting if and only if $\s$ is of finite co-dimension.
\end{thm}
\begin{proof}
The proof follows from the above lemma.
 \end{proof}

If the analytic Hilbert module $\clh$ in the above theorem is
$H^2(\D^n)$, then $P_{\q_j}M_z^*P_{\q_j^{\perp}}$ is a rank one
operator for all quotient modules $\clq_i$ of $H^2(\mathbb{D})$ and
$i = 1, \ldots, n$ (see Proposition 2.3 in \cite{JS1}). In
particular, for $\clh = H^2(\mathbb{D}^2)$, the submodule $\cls =
(\clq_1 \otimes \cle_2)^\perp$ is always essentially doubly
commuting. This result is due to Yang \cite{Y5}. For the Hardy space
$H^2(\mathbb{D}^n)$, Part (ii) was obtained by the third author in
\cite{JS1}.

The next two results becomes a useful variant of the above theorem.
\begin{cor}\label{Cor1}
For $n>2$, let $\s$ be a co-doubly commuting submodule of an
analytic Hilbert module $\Hil=\Hil_1\ot\cdots\ot\Hil_n$ and
$\q=\s^{\perp}(\cong \Hil/\s)$. Then the following are
equivalent.\\
\textup{(i)} $\s$ is essentially doubly commuting.\\
\textup{(ii)} $\s$ is of finite co-dimension.\\
\textup{(iii)} $\q$ is essentially normal.
\end{cor}

\begin{cor}\label{Cor2}
Let $\s$ be an essentially normal co-doubly commuting submodule of
an analytic Hilbert module $\Hil=\Hil_1\ot\cdots\ot\Hil_n$. If $\s$ is of
infinite co-dimension, then $n=2$.
\end{cor}

In the case $\clh = H^2(\mathbb{D}^n)$, both the Corollaries
\ref{Cor1} and \ref{Cor2} were obtained by the third author in
\cite{JS1}.

\section{Wandering subspaces and ranks of submodules}

In this section we investigate the existence of wandering subspace,
in the sense of Halmos \cite{H}, of a co-doubly commuting submodule
of $L^2_{a, \bm{\alpha}}(\D^n)$, and compute the rank of a co-doubly
commuting submodules of $H^2(\mathbb{D}^n)$. In particular, we
explicitly compute the rank of a co-doubly commuting submodule $\s$
of $H^2(\mathbb{D}^n)$ and prove that the rank of $\cls$ is not
greater than $n$.

We begin with the definition of wandering subspaces for submodules
of analytic Hilbert modules over $\mathbb{C}[\z]$.

Let $\s$ be a submodule of an analytic Hilbert module $\clh$ over
$\mathbb{C}[\z]$ and $\mathcal{W}\subseteq\s$ be a closed subspace.
Then $\clw$ is a \emph{wandering subspace} of $\cls$ if
\[\mathcal{W}\perp M_{\z}^{\bm{k}}\mathcal{W},\] for all $\bm{k} \in
\Nat^{n}\setminus \{0\}$ and \[\s = \bigvee_{\bm{k} \in
\mathbb{N}^n} M_z^{\bm{k}} \clw .\]

Let $\s$ be a submodule of $H^2(\mathbb{D})$, or $L^2_a(\D)$. Then
$\mathcal{W}=\s\ominus z\s$ is the wandering subspace of $\s$.
Moreover, the dimension of $\clw$ is always one for $\clh =
H^2(\mathbb{D})$ \cite{B}, and any value in the range $1, 2, \ldots,
\infty$, for $\clh = L^2_a(\mathbb{D})$ \cite{ABFP}. For a general
$n$, the existence of wandering subspaces of doubly commuting
submodules of $L^2_a(\D^n)$ is obtained in \cite{bergman1} and
\cite{bergman2}.

Now let $\s=(\q_1\ot\cdots\ot\q_n)^{\perp}$ be a co-doubly commuting
submodule of $L^2_{a, \bm{\alpha}}(\D^n)$, where
$\bm{\alpha}=(\alpha_1,\ldots, \alpha_n) \in \Nat^n$, and $\q_i$ is
a quotient module of $L^2_{a, \alpha_i}(\D)$, $i=1, \dots, n$. Let
\[\mathcal{W}_i=(\q_i^{\perp}\ominus z\q_i^{\perp}),\] be the wandering
subspace of $\q_i^{\perp}$ for $i=1,\dots,n$. Consider the set
\[
\mathcal{W}=\bigvee_{i=1}^n 1\ot\cdots\ot1\ot
\mathcal{W}_i\ot1\cdots\ot 1\subseteq \s.
\]
By virtue of Theorem \ref{co-doubly commuting}, it then follows
easily that \[\s = \bigvee_{\bm{k}\in \mathbb{N}^n} M_z^{\bm{k}}
\mathcal{W}.\]There is, however, no guarantee that $\mathcal{W}
\perp M_{z}^{\bm{k}} \mathcal{W}$ for all $\bm{k} \in
\Nat^{n}\setminus\{\mathbf{0}\}$. For instance, it is not
necessarily true that \[\langle 1\ot f_2\ot1\ot \cdots\ot1, f_1\ot
M_{z}1\ot 1\ot\cdots\ot 1\rangle = 0,\] for all $f_1
\in\mathcal{W}_1$ and $f_2 \in\mathcal{W}_2$.

\noindent However, if we further assume that $1\in\q_i$ for all
$i=1,\dots,n$, it then easily follows that $\mathcal{W}$ is a
wandering subspace of $\cls$. Thus we have the following result on
the existence of wandering subspaces of a class of co-doubly
commuting submodules of $L^2_{a, \bm{\alpha}}(\D^n)$.

\begin{thm}
Let $\bm{\alpha} = (\alpha_1, \ldots, \alpha_n) \in \Nat^n$ and
$\q_i$ be a quotient module of $L^2_{a, \alpha_i}(\D)$ and $1 \in
\clq_i$, $i=1,\ldots, n$. Then
\[
\mathcal{W}=\bigvee_{i=1}^n 1\ot\cdots\ot 1\ot
\mathcal{W}_i\ot1\cdots\ot 1
\]
is a wandering subspace of the co-doubly commuting submodule
$\s=(\q_1\ot\cdots\ot\q_n)^{\perp}$,
where\[\mathcal{W}_i=(\q_i^{\perp}\ominus z\q_i^{\perp}),\] for
$i=1,\dots,n$.
\end{thm}

We now study the rank of a co-doubly commuting submodule of an
analytic Hilbert module over $\mathbb{C}[\z]$. Recall that the
\textit{rank} of a Hilbert module $\clh$ over $\mathbb{C}[\z]$ is
the smallest cardinality of its generating sets \cite{DP}. More
precisely,
\[
\mbox{rank}(\clh) = \min_{E\in \clg(\clh)} \# E,\]where \[
\clg(\clh)=\{E\subseteq \Hil: \bigvee_{\bm{k} \in \mathbb{N}^n}
M_z^{\bm{k}} E =\Hil\}.
\]

%In the rest of this section we focus on co-doubly commuting
%submodules of the Hardy module $H^2(\D^n)$.

Let $\cls = \theta H^2(\mathbb{D})$ be a submodule of
$H^2(\mathbb{D})$ for some inner function $\theta \in
H^\infty(\mathbb{D})$ \cite{B}. Then \[\cls = \theta H^2(\mathbb{D})
= \bigvee_{m \geq 0} z^m E,\]where $E = \{\theta\}$. Consequently,
$\cls$ is of rank one. This is no longer true for Hardy space over
$\mathbb{D}^n$ and $n \geq 2$. As pointed out by Rudin \cite{R},
there exists a submodule $\s$ of $H^2(\mathbb{D})$ such that the
rank of $\s$ is not finite (see also \cite{III}, \cite{Se} and
\cite{SY}). We now consider the class of co-doubly commuting
submodules of $H^2(\D^n)$.

Let $\s$ be a non-trivial proper co-doubly commuting submodule of
$H^2(\D^n)$. Theorem \ref{co-doubly commuting} implies that there
exists non-zero quotient modules $\q_1,\dots,\q_n$ of $H^2(\D)$ such
that
\[
\s=(\q_1\ot\cdots\ot\q_n)^{\perp}= \sum_{i=1}^n
H^2(\D^{i-1})\ot\q_i^{\perp}\ot H^2(\D^{n-i}).
\]
Then there exists a natural number $1 \leq m \leq n$ such that
$\clq_{l_j} \neq H^2(\mathbb{D})$,r $j = 1, \ldots, m$. Let
$\theta_{l_j}$ be the inner function corresponding to the non-zero
submodule $\q_{l_j}^{\perp}$, that is,
\[\q_{l_j}^{\perp}=\theta_{l_j}H^2(\D). \quad \quad (j=1,\dots,m)\] Let
$E$ be the set of one variable inner functions corresponding to
$\{\theta_{l_j}\}_{j=1}^m$ over $\mathbb{D}^n$, that is,
\[E:=\{\Theta_{l_j}\in\s: \Theta_{l_j}=1\ot\cdots\ot
1\ot\underbrace{\theta_{l_j}}\limits_{l_j\text{-th}} \ot
1\ot\dots\ot 1, j=1,\dots,m\}. \] Then invoking again Theorem
\ref{co-doubly commuting} we conclude that
\[\bigvee_{\bm{k} \in \Nat^{n}} M_{z}^{\bm{k}} E = \s.\]Consequently, \[\mbox{rank} (\s) \le m.\] If,
in addition, we assume that $1\in\q_i$, for $1\le i\le n$ then
\[\Theta_{l_j} \in \mbox{ker} P_{\s} M_{z}^{* \bm{k}},\quad \quad (1\le j\le
m)\]for any non-zero $\bm{k} \in\Nat^{n}$. Then with a standard
argument we obtain $\mbox{rank} (\s)\ge m$ and hence $\mbox{rank}
(\s) = m$.

We summarize the results given above as follows.
\begin{thm}
Let $\s=(\q_1\ot\cdots\ot\q_n)^{\perp}$ be a co-doubly commuting
submodule of $H^2(\D^n)$. Then the rank of $\s$ is less than or
equal to the number of quotient modules $\q_i$ which are different
from $H^2(\D)$. Moreover, equality holds if $1\in\q_i$ for all $1\le
i\le n$.
\end{thm}

\section{Concluding Remarks}

It is worth stressing here that the results of this paper are based
on three essential assumptions on the Hilbert module $\clh$:

\begin{enumerate}
\item $\clh$ is a reproducing kernel Hilbert module over
$\mathbb{D}^n$. Moreover, the kernel function $K_{\clh}$ of $\clh$
is a product of one variable kernel functions over the unit disk
$\mathbb{D}$. That is, \[K_{\clh}(\z, \w) = \prod_{i=1}^n K_i(z_i,
w_i). \quad \quad (\z, \w \in \mathbb{D}^n)\]

\item $\clh$ is a standard reproducing kernel Hilbert module, that
is, there does not exists a pair of non-zero orthogonal quotient
modules of $\clh_{K_i} \subseteq \clo(\mathbb{D}, \mathbb{C})$,
where $\clh_{K_i}$ is the reproducing kernel Hilbert module
corresponding to the kernel $K_i$ and $i = 1, \ldots, n$.

\item $K^{-1}_{\clh}$ is a polynomial, or, that $\clh$ admits a
$\frac{1}{K}$-calculus, in the sense of Arazy and Englis.

\end{enumerate}

The purpose of the following example is to show that the conclusion
of Theorem \ref{analytic module} is false if we drop the assumption
that $\clh$ is standard.

Let $\Hil=\Hil_1\ot\cdots\ot\Hil_n$ be a reproducing kernel Hilbert
module over $\mathbb{C}[\z]$ such that $\Hil_1$ is not a standard
reproducing kernel Hilbert module over $\mathbb{C}[z]$. This implies
$\Hil_1$ has two orthogonal quotient modules $\q_1$ and $\q_1'$. Now
consider the following quotient module of $\Hil$
\[
\q= (\q_1\ot\q_2\ot\q_3\ot\cdots\ot\q_n)\oplus
(\q_1'\ot\q_2'\ot\q_3\ot\cdots\ot \q_n),
\]
for two different quotient modules $\q_2$ and $\q_2'$ of $\Hil_2$
and some quotient modules $\q_i$ of $\Hil_i$, $i=3,\dots,n$. Then it
is evident that $\q$ is a doubly commuting quotient module of $\Hil$
but it can not be represented in the form of tensor product of $n$
one variable quotient modules. Therefore one may ask the following
general question.

Is every doubly commuting quotient module of a Hilbert module over
$\mathbb{C}[\z]$ orthogonal sum of quotient modules each of which is
Hilbert tensor product of one variable quotient modules?

\noindent\textbf{Acknowledgment:} The first and second authors are
grateful to Indian Statistical Institute, Bangalore Centre for warm
hospitality. The first author also thanks NBHM for financial
support.

\end{document}